\documentclass{amsart}

\usepackage{amssymb,mathtools}
\usepackage[numbers,sort&compress]{natbib}
\usepackage{microtype}
\usepackage{enumitem}
\usepackage[hidelinks]{hyperref}

\numberwithin{equation}{section}
\theoremstyle{plain}
\newtheorem{theorem}{Theorem}[section]

\newtheorem{lemma}[theorem]{Lemma}
\newtheorem{corollary}[theorem]{Corollary}
\theoremstyle{remark}
\newtheorem{remark}[theorem]{Remark}

\newcommand{\HH}{\mathbb H}
\newcommand{\RR}{\mathbb R}

\title[Radial solutions to a hyperbolic H\'enon equation]
{Structure of radial solutions to a H\'enon type equation with
exponential nonlinearity on the hyperbolic space}

\author{Yahui Jiang}
\address{Y. Jiang, School of Mathematical Sciences, Key Laboratory of MEA
(Ministry of Education) \& Shanghai Key Laboratory of PMMP,
East China Normal University, Shanghai 200241, People's Republic of China}
\email{52275500028@stu.ecnu.edu.cn}

\author{Xianmei Zhou}
\address{X. Zhou, School of Mathematical Sciences, Zhejiang Normal University,
Jinhua 321004, People's Republic of China}
\email{xmzhou@zjnu.edu.cn}

\thanks{Corresponding author: Xianmei Zhou.}

\begin{document}

\begin{abstract}
In this paper, we study the separation and stability properties of
radial solutions to a H\'enon-type equation with exponential
nonlinearity on the hyperbolic space. By transforming the radial
hyperbolic equation into a weighted Euclidean equation and applying
well-established separation results in Euclidean space, we classify the solution structures or establish sharp alternatives for radial solutions throughout the full range $n\ge2$ and $\alpha>-2$.  We also give an affirmative answer to the open question recently posed by Huang and Zhao \cite{HZ}. 
\end{abstract}

\maketitle

\section{Introduction and main results}
In this paper, we investigate the structure of radial solutions to the
H\'enon type equation with exponential nonlinearity on the hyperbolic space
\begin{equation}\label{eq1}
    -\Delta_{\HH^n}u=(\sinh r)^\alpha\exp(u)
    \quad\text{in }\HH^n,
\end{equation}
where $n\ge2$, $\alpha>-2$, $r$ denotes the geodesic distance from the origin in $\HH^n$, and $\Delta_{\HH^n}$ is the Laplace--Beltrami operator on $\HH^n$.

\medskip
A function $u$ is called a weak solution
to \eqref{eq1} if $$u\in H_{\mathrm{loc}}^1(\HH^n), \qquad (\sinh r)^{\alpha}\exp(u)\in L_{\mathrm{loc}}^1(\HH^n),$$ and \eqref{eq1} is satisfied in the weak sense, that is,
\begin{align*}\label{weak}
\int_{\HH^n}\left(
\langle\nabla_{\HH^n}u,\nabla_{\HH^n}\varphi\rangle
-(\sinh r)^\alpha\exp(u)\varphi
\right)\,dV_{\HH^n}=0,\quad\forall\varphi\in C_{c}^\infty(\HH^n).
\end{align*}
A solution $u$ of \eqref{eq1} is said to be stable if
\begin{align*}
    Q_{u}(\varphi):=\int_{\HH^n}\left(|\nabla_{\HH^n}\varphi|_{\HH^n}^2-(\sinh r)^{\alpha}\exp(u)\varphi^2\right)\,dV_{\HH^n}\ge0,\quad\forall\varphi\in C_{c}^\infty(\HH^n),
\end{align*}
where 
\begin{align*}
|\nabla_{\HH^n}\varphi|_{\HH^n}^2=\langle\nabla_{\HH^n}\varphi,\nabla_{\HH^n}\varphi\rangle_{\HH^n}=|\partial_r\varphi|^2+(\sinh r)^{-2}|\nabla_{\mathbb{S}^{n-1}}\varphi|^2.
\end{align*}
For each $\beta\in\RR$, a radial solution of \eqref{eq1} with
 $u(0)=\beta$ satisfies the  initial value
problem
\begin{equation}\label{radial1}
\begin{cases}
    u''(r)+(n-1)\coth r\,u'(r)=-(\sinh r)^{\alpha}\exp(u(r)),\quad r>0,\\
    u(0)=\beta.
    \end{cases}
\end{equation}
This problem \eqref{radial1} admits a unique solution
$u_{\beta}\in C([0,\infty))\cap C^2((0,\infty))$ satisfying
\begin{equation*}
    \lim_{r\to0^+}(\sinh r)^{n-1}u_\beta'(r)=0.
\end{equation*}
When $\alpha\geq 0$,
$u_\beta$ is a classical radial solution of \eqref{eq1}. Moreover,
$u_\beta$ is a radial weak solution of \eqref{eq1} for every
$\alpha>-2$. Indeed, integrating \eqref{radial1}, we obtain 
\begin{equation}\label{radial:flux}
    (\sinh r)^{n-1}u_\beta'(r)
    =-\int_0^r(\sinh s)^{n-1+\alpha}\exp(u_\beta(s))\,ds.
\end{equation}
Since $u_\beta$ is bounded near the origin and
$\sinh r\sim r$ as $r\to0$, it follows by \eqref{radial:flux} that
\begin{align*}
    u_\beta'(r)=O(r^{\alpha+1}),\qquad
    u_\beta(r)-\beta=O(r^{\alpha+2})
    \quad\text{as }r\to0.
\end{align*}
Thus,
$
    u_\beta\in H_{\rm loc}^1(\HH^n)$ and $
    (\sinh r)^\alpha \exp(u_\beta(r))\in L_{\rm loc}^1(\HH^n).
$

\medskip
In the Euclidean setting, the H\'enon type equation with exponential nonlinearity
\begin{equation}\label{Euc}
-\Delta u=|x|^\alpha\exp(u)
\quad\text{in }\RR^n
\end{equation}
with $n\ge2$ and $\alpha>-2$ has been extensively studied from the viewpoints of stability and
separation of solutions. For the autonomous case $(\alpha=0)$, the structure of classical radial solutions is well understood: any two distinct radial solutions intersect infinitely many times for $3\le n\le9$, whereas they separate when $n\ge10$; see, \cite{T}. Farina \cite{F} proved that \eqref{Euc} has no stable classical solution for $2\le n\le9$, while the existence of a stable radial solution for $n\ge10$ follows from the analysis performed in \cite{JL}. For the nonautonomous case $(\alpha\neq0)$, Wang and Ye \cite{WY} established a sharp Liouville theorem for stable weak solutions. More precisely, they showed that no stable weak solution exists for $2\le n<10+4\alpha$, whereas radial stable weak solutions always exist for $n\ge10+4\alpha$. It revealed that $10+4\alpha$ is the critical dimension for stability in the Euclidean setting. The structure of radial solutions to the more general equations $-\Delta u=K(|x|)\exp(u)$ has also been studied in \cite{B,BN,H}.

\medskip
Over the past decade, there has been a growing interest in the stability and separation properties of classical radial solutions to semilinear elliptic equations on the hyperbolic space $\HH^n$; see \cite{BFGR,BFG,H,H1,H2,H3,IT}. It is worth noting that Ioku and Toyoshima \cite{IT} studied the H\'enon type equation $-\Delta_{\HH^n}u=(\sinh r)^\alpha |u|^{p-1}u$ in $\HH^n$ for $n\ge2$ and $\alpha>0$. They introduced a transformation that maps radial solutions of an equation in $\HH^n$ to radial solutions of a corresponding weighted equation in $\RR^m$, without requiring $m=n$. By combining this transformation with known results in the Euclidean setting, they further obtained a detailed classification of positive radial solutions.

\medskip
For \eqref{eq1}, Huang and Zhao \cite{HZ} obtained the following stability results for classical solutions.
\begin{theorem}\cite[Theorem~1.3--1.5]{HZ}\label{HZthm}
    Let $n>2$ and $\alpha>0$.
    \begin{enumerate}[label=\textnormal{(\roman*)},leftmargin=*,nosep]
        \item If $n<1+4\alpha$, then \eqref{eq1} has no stable solution.
        \item If $n>1+4\alpha$, then there exists $\beta_{0}(n,\alpha)\in\RR$ such that $u_\beta$ is stable for $\beta\le\beta_{0}(n,\alpha)$.
        \item If $n<10+4\alpha$, then there exists a large $\beta_1$ such that $u_\beta$ is unstable for $\beta>\beta_1$; while for $n\ge10+4\alpha$, $u_\beta$ is stable for $\beta>\beta_1$.
    \end{enumerate}
\end{theorem}

Theorem~\ref{HZthm} indicates that $1+4\alpha$ is a critical threshold dimension for stability in $\HH^n$, which is different from the Euclidean setting. In particular, they raised an open question: \emph{Is $u_\beta$ stable for every $\beta\in\RR$, when $n\ge10+4\alpha$}? 
Furthermore, the results in \cite{HZ} do not cover the singular
weight regime $\alpha\in(-2,0)$ or the two-dimensional case $n=2$.
The separation and stability properties in these regimes therefore
remain to be clarified.

Motivated by these questions, we investigate the separation and stability property of radial solutions to \eqref{eq1}. To state our separation results clearly, we define the following classification:
\begin{enumerate}[label=(\Roman*),leftmargin=*,nosep]
    \item For any $\beta,\gamma\in\RR$ with
    $\beta\ne\gamma$, two solutions $u_\beta$ and $u_\gamma$
    intersect at least once on $(0,\infty)$.

    \item For any $\beta,\gamma\in\RR$ with
    $\beta\ne\gamma$, two solutions $u_\beta$ and $u_\gamma$
    cannot intersect each other on $(0,\infty)$.

    \item There exists $\beta^*\in\RR$ such that, for any
    $\beta,\gamma\le\beta^*$ with $\beta\ne\gamma$, two solutions
    $u_\beta$ and $u_\gamma$ cannot intersect each other on $(0,\infty)$, and that for any $\beta,\gamma>\beta^*$ with
    $\beta\ne\gamma$, two solutions
    $u_\beta$ and $u_\gamma$ intersect at least once on $(0,\infty)$.
\end{enumerate}
We refer to (I), (II), and (III) as the intersection structure,
the separation structure, and the partial separation structure,
respectively.

\medskip
Our main separation results are as follows.

\begin{theorem}\label{thm1}
Let $n\ge3$ and $\alpha>-2$. Then the solution family of \eqref{eq1} has
structure {\rm (I)} if $n<1+4\alpha$, structure {\rm (III)} if
$1+4\alpha<n<10+4\alpha$, and structure {\rm (II)} if
$n\ge10+4\alpha$. If $n=1+4\alpha$, then it has either
structure {\rm (I)} or {\rm (III)}.
\end{theorem}

%The two-dimensional case is different because the transformed weight has a non-power-type behavior near the origin.
\begin{theorem}\label{thm2}
Let $n=2$ and $\alpha>-2$. Then the solution family of \eqref{eq1} has structure {\rm (I)} if $\alpha>1/4$, structure {\rm (III)} if $\alpha=1/4$, and either structure
{\rm (II)} or {\rm (III)} if $-2<\alpha<1/4$.
\end{theorem}

The above separation results, together with the relation between
separation and stability for the transformed Euclidean problem,
lead to the following stability classification.

\begin{corollary}\label{cor1}
Let $n\ge3$ and $\alpha>-2$.
\begin{enumerate}[label=\textnormal{(\roman*)},leftmargin=*,nosep]
    \item If $n<1+4\alpha$, then $u_\beta$ is unstable for any
    $\beta\in\RR$.

    \item If $1+4\alpha<n<10+4\alpha$, then there exists
    $\beta^*\in\RR$ such that $u_\beta$ is stable for any
    $\beta\le\beta^*$, and $u_\beta$ is unstable for any $\beta>\beta^*$.

    \item If $n\ge10+4\alpha$, then $u_\beta$ is stable for any $\beta\in\RR$.
    
    \item If $n=1+4\alpha$, then either any radial solution is unstable, or there exists $\beta^*\in\RR$ such that $u_\beta$ is stable for $\beta\le\beta^*$ and unstable for $\beta>\beta^*$.
\end{enumerate}
\end{corollary}

\begin{corollary}\label{cor2}
Let $n=2$ and $\alpha>-2$.
\begin{enumerate}[label=\textnormal{(\roman*)},leftmargin=*,nosep]
    \item If $\alpha>1/4$, then $u_\beta$ is unstable for every
    $\beta\in\RR$.

    \item If $\alpha=1/4$, then there exists $\beta^*\in\RR$
    such that $u_\beta$ is stable for $\beta\le\beta^*$ and
    unstable for $\beta>\beta^*$.

    \item If $-2<\alpha<1/4$, then either every radial solution
    is stable, or there exists $\beta^*\in\RR$ such that
    $u_\beta$ is stable for $\beta\le\beta^*$ and unstable for
    $\beta>\beta^*$.
\end{enumerate}
\end{corollary}

\begin{remark} Corollary~\ref{cor1} (iii) gives an affirmative
answer to the question raised by Huang and Zhao  \cite{HZ}. Moreover, in the range $\alpha>0$ and $1+4\alpha<n<10+4\alpha$, Theorem~\ref{HZthm} and Corollary~\ref{cor1} (ii) give $\beta^*\in[\beta_{0}(n,\alpha),\beta_1]$.
\end{remark}

The paper is organized as follows. 
In Section~2, we introduce the transformation from
the radial hyperbolic equation to a weighted Euclidean equation and prove
Theorems~\ref{thm1} and \ref{thm2}.
In Section~3, we give the proof of Corollaries~\ref{cor1} and~\ref{cor2}.
Finally, the Euclidean separation results used in Section~2 are given in Appendix.

\section{Proofs of Theorems~\ref{thm1} and~\ref{thm2}}
Fix $n\ge2$, an integer $m\ge3$, and $\alpha>-2$. Define
\begin{align*}
    T(t)=\int_t^\infty\frac{ds}{(\sinh s)^{n-1}},\qquad t>0.
\end{align*}
Since $(\sinh t)^{1-n}\in L^1(1,\infty)$, the function
$T\in C^\infty((0,\infty))$ is well defined and positive. As $t\to0^+$,
\begin{equation}\label{T0}
    T(t)\sim
    \begin{cases}
        \log(1/t), &\quad\text{if }n=2,\\[1mm]
        \dfrac{t^{2-n}}{n-2}, &\quad\text{if }n\ge3,
    \end{cases}
\end{equation}
whereas
\begin{equation}\label{T:asyinfty}
    T(t)\sim\frac{2^{n-1}}{n-1}\exp((1-n)t)
    \qquad\text{as }t\to\infty.
\end{equation}
Define the transformation $\phi:(0,\infty)\to(0,\infty)$ by
\begin{align*}
    \phi(t)=\bigl[(m-2)T(t)\bigr]^{1/(2-m)}, 
\end{align*}
then
\begin{equation}\label{phi:der}
    \phi'(t)=\frac{\phi^{m-1}(t)}{(\sinh t)^{n-1}}>0,
    \qquad t>0.
\end{equation}
Furthermore, it follows from \eqref{T0} and \eqref{T:asyinfty} that, 
\begin{equation}\label{phi:asy0}
    \phi(t)\sim
    \begin{cases}
        \bigl[(m-2)\log(1/t)\bigr]^{-1/(m-2)}, &\quad\text{if }n=2,\\[1mm]
        \displaystyle
        \left(\frac{n-2}{m-2}\right)^{1/(m-2)}
        t^{(n-2)/(m-2)}, &\quad\text{if }n\ge3,
    \end{cases}
    \qquad\text{as }t\to0^+.
\end{equation}
and 
\begin{equation}\label{phi:asyinfty}
    \phi(t)\sim
    \left[\frac{n-1}{2^{n-1}(m-2)}\right]^{1/(m-2)}
    \exp\left(\frac{n-1}{m-2}t\right)
    \qquad\text{as }t\to\infty.
\end{equation}
Thus $\phi\in C^\infty((0,\infty))$ is a strictly increasing
bijection from $(0,\infty)$ onto $(0,\infty)$.

For a solution $u$ of \eqref{radial1}, set
$v(r)=u(t)$, where $r=\phi(t)$. Then $v$ satisfies
\begin{equation}\label{eq2}
    \begin{cases}
        \displaystyle
        v''(r)+\frac{m-1}{r}v'(r)+K(r)\exp(v(r))=0,
        & r>0,\\
        v(0)=\beta,
    \end{cases}
\end{equation}
where
\begin{align*}
    K(r)=
    \frac{\bigl(\sinh(\phi^{-1}(r))\bigr)^{2(n-1)+\alpha}}
    {r^{2(m-1)}}.
\end{align*}
Moreover, by \eqref{phi:der}, we have 
\begin{align*}
    r^{m-1}v'(r)
    =(\sinh t)^{n-1}u'(t)\longrightarrow0
    \qquad\text{as }r\to0^+.
\end{align*}
Thus $v$ is the regular radial solution of \eqref{eq2}. When
necessary, we write $v=v_\beta$ to indicate its dependence on
$\beta$. Clearly, $K\in C^\infty((0,\infty))$ and $K(r)>0$ for
$r>0$.

We next determine the asymptotic behavior of $K$ near the origin.
If $n=2$, then $T(t)=\log\coth(t/2)$. Since
$r^{m-2}=((m-2)T(t))^{-1}$, we have
\begin{align*}
    t
    &=2\operatorname{arctanh}\left(
    \exp\left\{-\frac{1}{(m-2)r^{m-2}}\right\}\right)\\
    &=2\exp\left\{-\frac{1}{(m-2)r^{m-2}}\right\}(1+o(1))
    \qquad\text{as }r\to0^+.
\end{align*}
For $n\ge3$, the corresponding power-type asymptotic follows from
\eqref{phi:asy0}, that is, 
\begin{equation}\label{Kasy0}
    K(r)\sim
    \begin{cases}
        \displaystyle
        2^{2+\alpha}r^{-2(m-1)}
        \exp\left(
        -\frac{2+\alpha}{m-2}r^{-(m-2)}
        \right), &\quad\text{if }n=2,\\[2mm]
        \displaystyle
        k_0 r^{\ell_0}, &\quad\text{if }n\ge3,
    \end{cases}
    \qquad\text{as }r\to0^+,
\end{equation}
where
\begin{align*}
    \ell_0
    =\frac{(m-2)(\alpha+2)}{n-2}-2>-2,
    \qquad
    k_0
    =\left(\frac{m-2}{n-2}\right)^{
    \frac{2(n-1)+\alpha}{n-2}}>0,
\end{align*}
since $n\ge3$, $m\ge3$ and $\alpha>-2$. Consequently, $rK\in L^1(0,1)$ and $\lim_{r\to0^+}r^2K(r)=0$ for any $n\ge2$ and $\alpha>-2$.

Similarly, we obtain from  \eqref{phi:asyinfty} that 
\begin{equation}\label{Kinfty}
    \lim_{r\to\infty}r^{-\ell_\infty}K(r)=k_\infty,
\end{equation}
where 
\begin{align*}
    \ell_\infty
    =\frac{m-2}{n-1}\alpha-2,
    \qquad
    k_\infty
    =\left(\frac{m-2}{n-1}\right)^{
    2+\frac{\alpha}{n-1}}>0.
\end{align*}

We are now in a position to prove Theorems~\ref{thm1} and
\ref{thm2}. Since $\phi$ is a strictly increasing bijection
independent of $\beta$, two solutions $u_\beta$ and $u_\gamma$ of \eqref{radial1}
intersect if and only if the corresponding solutions $v_\beta$ and
$v_\gamma$ of \eqref{eq2} intersect. Therefore, it suffices to establish the claimed
separation structures for the family $\{v_\beta\}$.

\medskip

\textbf{Proof of Theorem~\ref{thm1}.}
We divide the proof into three steps.

\textbf{Step 1.} We prove that (III) holds for $1+4\alpha<n<10+4\alpha$, and that either (I) or (III) holds for $n=1+4\alpha$. Indeed, since $m\ge3$, a direct calculation shows that $m<10+4\ell_0$ if and only if $n<10+4\alpha$. Thus, by \eqref{Kasy0} and Theorem~\ref{BNthm1}, either (I) or (III) holds for $3\le n<10+4\alpha$.

We next exclude structure (I) when $n>1+4\alpha$ by applying
Theorem~\ref{BNcor1}.  Choose $\ell=(m-10)/4>-2$ and set $t=\phi^{-1}(r)$. Then $t'(r)=(\phi'(t))^{-1}$ and
\begin{align*}
    \frac{rK'(r)}{K(r)}
    &=(2(n-1)+\alpha)\frac{\phi(t)\coth t}{\phi'(t)}-2(m-1)\\
    &=(2(n-1)+\alpha)(m-2)(\sinh t)^{n-2}\cosh t\,T(t)-2(m-1).
\end{align*}
Since $t\to\infty$ as $r\to\infty$, it follows from \eqref{T:asyinfty} that
\begin{align*}
    \lim_{r\to\infty}\frac{rK'(r)}{K(r)}=\frac{m-2}{n-1}\alpha-2.
\end{align*}
Consequently,
\begin{align*}
    \lim_{r\to\infty}\left(\frac{rK'(r)}{K(r)}-\ell\right)=\frac{m-2}{n-1}\alpha-2-\frac{m-10}{4}=
    \frac{(m-2)\bigl(4\alpha-(n-1)\bigr)}
         {4(n-1)}<0
\end{align*}
when $n>1+4\alpha$. Hence, for some $r_0>0$,
\begin{align*}
    \frac{d}{dr}\bigl(r^{-\ell}K(r)\bigr)
    =r^{-\ell-1}K(r)\left(\frac{rK'(r)}{K(r)}-\ell\right)<0,
    \qquad r\ge r_0.
\end{align*}
Therefore, Theorem~\ref{BNcor1} yields either (II) or (III) for
$n>1+4\alpha$. Combined with the
preceding alternative, we conclude that (III) holds for
$1+4\alpha<n<10+4\alpha$. When $n=1+4\alpha$, the first conclusion
yields either (I) or (III) holds.

\textbf{Step 2.} We prove that (I) holds if $n<1+4\alpha$. In this range, $\alpha>(n-1)/4>0$, and so $\ell_\infty>-2$. Moreover, $m<10+4\ell_\infty$ is equivalent to $n<1+4\alpha$. Thus, by \eqref{Kinfty} and Theorem~\ref{BNthmC}, (I) follows.

\textbf{Step 3.} We prove (II) holds when $n\ge10+4\alpha$. Choose $\ell=\ell_0$ and set $t=\phi^{-1}(r)$. Then
\begin{align*}
    r^{-\ell}K(r)=\frac{(\sinh t)^{2(n-1)+\alpha}}{r^{\frac{(m-2)(2(n-1)+\alpha)}{n-2}}}
    =\bigl((m-2)(\sinh t)^{n-2}T(t)\bigr)^{
    \frac{2(n-1)+\alpha}{n-2}}.
\end{align*}
Moreover,
\begin{align*}
    T(t)=\int_{t}^{\infty}\frac{ds}{(\sinh s)^{n-1}}<\frac{1}{\cosh t}\int_t^\infty
    \frac{\cosh s}{(\sinh s)^{n-1}}\,ds
    =\frac{1}{(n-2)(\sinh t)^{n-2}\cosh t},
\end{align*}
and hence
\begin{align*}
\begin{aligned}
    \frac{d}{dt}(\sinh t)^{n-2}T(t)&=(n-2)(\sinh t)^{n-3}\cosh t\,T(t)-(\sinh t)^{n-2}(\sinh t)^{1-n}\\
    &=\frac{(n-2)(\sinh t)^{n-2}\cosh t\,T(t)-1}{\sinh t}<0.
 \end{aligned}   
\end{align*}
Since $t'(r)=(\phi'(t))^{-1}>0$, the function
$F(r):=r^{-\ell_0}K(r)$ is strictly decreasing on $(0,\infty)$.
Consequently,
$
    K_{\ell_0}(r)=\inf_{0<s<r}s^{-\ell_0}K(s)=F(r).
$
Moreover, $m\ge10+4\ell_0$ if and only if
$n\ge10+4\alpha$, and in this case
$
    \delta=\frac{m-2}{4(2+\ell_0)}\ge1.
$
Thus $$r^{-\ell_0}K(r)\le\delta K_{\ell_0}(r),$$ so
\eqref{nonincreasing} holds. Theorem~\ref{Bthm1} now yields (II) for
$n\ge10+4\alpha$. This completes the proof.\qed

\medskip

\textbf{Proof of Theorem~\ref{thm2}.}
For $n=2$, the argument in Step~2 of the proof of
Theorem~\ref{thm1} shows that (I) holds when $\alpha>1/4$.
If $-2<\alpha<1/4$, the calculation in Step~1 of the same
proof, with $\ell=(m-10)/4$, shows that
$r^{-\ell}K(r)$ is eventually decreasing. It follows from
Theorem~\ref{BNcor1} that either ${\rm (II)}$ or
${\rm (III)}$ holds.

It remains to consider the critical case $\alpha=1/4$. Set
$
    \ell=\frac{m-10}{4},
$
so that $m=10+4\ell$ and $\ell>-2$. Since $n=2$, we have $
    T(t)=\log\coth\frac{t}{2}.
$
Writing $t=\phi^{-1}(r)$ and using
$r^{m-2}=((m-2)T(t))^{-1}$, we obtain
\begin{equation}\label{criticalK}
\begin{aligned}
    r^{-\ell}K(r)
    &=(m-2)^{9/4}\bigl(\sinh t\,T(t)\bigr)^{9/4}.
\end{aligned}
\end{equation}
We claim that the right-hand side is strictly increasing in $t$.
Indeed,
\begin{align*}
    \frac{d}{dt}\bigl(\sinh t\,T(t)\bigr)
    =\cosh t\,T(t)-1>0.
\end{align*}
To verify the last inequality, observe that
\begin{align*}
    \left(T(t)-\frac{1}{\cosh t}\right)'
    =-\frac{1}{\sinh t\cosh^2t}<0,
    \qquad
    \lim_{t\to\infty}\left(T(t)-\frac{1}{\cosh t}\right)=0.
\end{align*}
Thus $T(t)>1/\cosh t$ for all $t>0$. Since $t=\phi^{-1}(r)$
is strictly increasing in $r$, \eqref{criticalK} shows that
$r^{-\ell}K(r)$ is strictly increasing. Moreover,
$\sinh t\,T(t)\to0$ as $t\to0^+$ by \eqref{T0}, while
$\sinh t\,T(t)\to1$ as $t\to\infty$ by
\eqref{T:asyinfty}. Therefore,
\begin{equation}\label{criticalKlimits}
    \lim_{r\to0^+}r^{-\ell}K(r)=0,
    \qquad
    \lim_{r\to\infty}r^{-\ell}K(r)=k_\infty
    =(m-2)^{9/4}.
\end{equation}

We next verify the following integrability condition
 \begin{equation}\label{criticalintegral}
    \int_1^\infty
    \bigl(k_\infty-r^{-\ell}K(r)\bigr)
    r^{(m-4)/2}\,dr<\infty.
\end{equation}
 In fact, we have 
\begin{align*}
    T(t)&=2e^{-t}+\frac{2}{3}e^{-3t}+O(e^{-5t}),\\
    \sinh t\,T(t)&=1-\frac{2}{3}e^{-2t}+O(e^{-4t})
    \qquad\text{as }t\to\infty.
\end{align*}
It follows from \eqref{criticalK} that
$
    k_\infty-r^{-\ell}K(r)=O(e^{-2t}),
$
Notice that 
$
r^{m-2}=\frac{1}{(m-2)T(t)}
\sim\frac{e^t}{2(m-2)},
$
we have
\[
\bigl(k_\infty-r^{-\ell}K(r)\bigr)r^{(m-4)/2}
=O\left(r^{-(3m-4)/2}\right)
\qquad\text{as }r\to\infty,
\]
and hence
\eqref{criticalintegral} holds. Theorem~\ref{Bcritical} then
shows that either ${\rm (II)}$ or ${\rm (III)}$ holds.

We finally exclude ${\rm (II)}$ by a scaling argument. Set
$
    q=2+\alpha=\frac94$, $
     \lambda_\beta=e^{-\beta/q}=e^{-4\beta/9}$,
and, for $s\ge0$, define the rescaled functions
\begin{align*}
    w_\beta(s)=u_\beta(\lambda_\beta s)-\beta,\qquad
    z_\beta(s)=u_{\beta+1}(\lambda_\beta s)-\beta.
\end{align*}
Since $\lambda_\beta^q e^\beta=1$, a direct computation shows that
both $w_\beta$ and $z_\beta$ satisfy
\begin{equation}\label{rescaledcritical}
    y''+\lambda_\beta\coth(\lambda_\beta s)y'
    +\left(\frac{\sinh(\lambda_\beta s)}{\lambda_\beta}
     \right)^{1/4}e^y=0,
\end{equation}
with the initial conditions
\begin{align*}
    w_\beta(0)=0,\quad z_\beta(0)=1,\qquad
    w_\beta'(0)=z_\beta'(0)=0.
\end{align*}

Since \eqref{rescaledcritical} is singular at $s=0$, we rewrite
it in integral form.
Let
$
    A_\lambda(s)=\frac{\sinh(\lambda s)}{\lambda}.
$
A solution of \eqref{rescaledcritical} with $y(0)=a$ satisfies
the integral equation
\begin{equation}\label{criticalintegraleq}
    y(s)=a-\int_0^s\frac{1}{A_\lambda(\tau)}
    \int_0^\tau A_\lambda(\sigma)^{5/4}e^{y(\sigma)}
    \,d\sigma\,d\tau.
\end{equation}
For every $R>0$,  define $A_\lambda(s)/s=1$ at $s=0$.
We have $A_\lambda(s)/s\to1$ uniformly on $[0,R]$ as
$\lambda\to0^+$. Moreover, for $0<\lambda\le1$ and
$0\le s\le R$, there is a constant $C_R>0$ such that
\begin{align*}
    s\le A_\lambda(s)\le C_Rs.
\end{align*}
Hence we have 
\begin{align*}
    y(s)\le a,
    \qquad
    |y'(s)|\le C_Re^a s^{5/4},
    \qquad 0\le s\le R.
\end{align*}
Integrating the estimate for $y'$,  it follows that 
\begin{align*}
    |y(s)-a|\le C_Re^a s^{9/4},\qquad 0\le s\le R.
\end{align*}
Applying the above estimates with $a=0,1$ and
$\lambda=\lambda_\beta$, we see that the families
$\{w_\beta\}$ and $\{z_\beta\}$ are uniformly bounded and
equicontinuous on every compact interval. Hence, by the
Arzelà--Ascoli theorem, every sequence $\beta_j\to\infty$
has a subsequence along which both $w_{\beta_j}$ and
$z_{\beta_j}$ converge uniformly on compact intervals.
Passing to the limit in \eqref{criticalintegraleq}, by dominated convergence shows that
each subsequential limit $W_a$, with $a\in\{0,1\}$, satisfies
\begin{equation}\label{Euclideancritical}
    W_a''+\frac1sW_a'+s^{1/4}e^{W_a}=0,\qquad s>0,
    \quad W_a(0)=a,\quad W_a'(0)=0.
\end{equation}
The unique solution of \eqref{Euclideancritical} is
\[
W_a(s)
=a-2\log\left(1+\frac{e^a}{2q^2}s^q\right),
\qquad q=\frac94.
\]
Thus 
$
w_\beta\rightarrow W_0$, $
z_\beta\rightarrow W_1
$
uniformly on compact subsets of \([0,\infty)\) as $\beta\to\infty$.

Let $D=W_1-W_0$. since $
D(0)=1$, and $\lim_{s\to\infty}D(s)=-1$, choose $S>0$ such that $D(S)<0$. We have 
$
z_\beta(S)-w_\beta(S)<0,
$
for all sufficiently large $\beta$, while
$z_\beta(0)-w_\beta(0)=1$. By continuity, there exists
$s_\beta\in(0,S)$ such that
$
z_\beta(s_\beta)=w_\beta(s_\beta),
$
or equivalently,
$
u_{\beta+1}(\lambda_\beta s_\beta)
=u_\beta(\lambda_\beta s_\beta).
$
Thus two distinct radial solutions intersect, excluding
structure ${\rm (II)}$. Together with the alternative obtained
from Theorem~\ref{Bcritical}, this proves that structure
${\rm (III)}$ holds when $\alpha=1/4$. \qed

\section{Proofs of Corollaries~\ref{cor1} and~\ref{cor2}}
We first show the following radial stability criterion.
\begin{lemma}\label{lem:radial-stability}
Let \(n\ge2\), \(\alpha>-2\), and let \(u\) be a radial solution of \eqref{eq1}. Then \(u\) is stable if and only if
\begin{equation}\label{eq:radial-stability}
\int_0^\infty
\left\{|\chi'(r)|^2-(\sinh r)^\alpha e^{u(r)}\chi(r)^2\right\}
(\sinh r)^{n-1}\,dr\ge0,
\quad\forall\chi\in C_c^\infty((0,\infty)).
\end{equation}
\end{lemma}
\textbf{Proof.} Necessity follows by taking radial test functions
$\varphi(x)=\chi(r(x))$ in the stability inequality. Conversely, suppose that
\eqref{eq:radial-stability} holds, and let
$\varphi\in C_c^\infty(\HH^n)$. Choose $h\in C^\infty(\RR)$ with
$0\le h\le1$, $h=0$ on $(-\infty,0]$, and $h=1$ on
$[1,\infty)$. For $0<\varepsilon<1/2$, define
\begin{equation*}
\eta_\varepsilon(r)
=h\left(\frac{\log\bigl(r/\varepsilon^2\bigr)}{|\log\varepsilon|}\right),
\qquad
\varphi_\varepsilon(x)=\eta_\varepsilon(r(x))\varphi(x).
\end{equation*}
Then $\eta_\varepsilon=0$ on $[0,\varepsilon^2]$, $\eta_\varepsilon=1$ on $[\varepsilon,\infty)$, and
\begin{equation}\label{eq:pole-cutoff-energy}
\int_0^\infty
|\eta_\varepsilon'(r)|^2
(\sinh r)^{n-1}\,dr
\leq
\frac{C}{|\log\varepsilon|^2}
\int_{\varepsilon^2}^{\varepsilon}\frac{dr}{r}
\leq
\frac{C}{|\log\varepsilon|}
\rightarrow0\quad\text{as }\varepsilon\to0^+.
\end{equation}
Since $\varphi_\varepsilon-\varphi$ is supported in
$B_\varepsilon$ and
$
    \nabla_{\HH^n}(\varphi_\varepsilon-\varphi)
    =(\eta_\varepsilon-1)\nabla_{\HH^n}\varphi
      +\varphi\nabla_{\HH^n}\eta_\varepsilon,
$
it follows from  \eqref{eq:pole-cutoff-energy}  that
$\varphi_\varepsilon\to\varphi$ strongly in $H^1(\HH^n)$.
Since $(\sinh r)^{\alpha}\exp(u)\in L_{\mathrm{loc}}^1(\HH^n)$,
the dominated convergence theorem yields
\begin{equation*}
\int_{\HH^n}(\sinh r)^{\alpha}\exp(u)\varphi_\varepsilon^2\,dV_{\HH^n}
\rightarrow
\int_{\HH^n}(\sinh r)^{\alpha}\exp(u)\varphi^2\,dV_{\HH^n}\quad\text{as }\varepsilon\to0^+.
\end{equation*}
For each fixed $\theta\in\mathbb S^{n-1}$,
$\varphi_\varepsilon(r,\theta)\in C_c^\infty((0,\infty))$.
Applying \eqref{eq:radial-stability} to
$\varphi_\varepsilon(r,\theta)$ and integrating over
$\mathbb S^{n-1}$, we obtain
\begin{equation*}
\begin{aligned}
0&\le
\int_{\HH^n}
\left[|\partial_r\varphi_\varepsilon|^2
-(\sinh r)^{\alpha}e^u\varphi_\varepsilon^2\right]\,dV_{\HH^n}\\
&\le
\int_{\HH^n}
\left[|\nabla_{\HH^n}\varphi_\varepsilon|_{\HH^n}^2
-(\sinh r)^{\alpha}e^u\varphi_\varepsilon^2\right]\,dV_{\HH^n}.
\end{aligned}
\end{equation*}
Passing to the limit as $\varepsilon\to0^+$ shows that
$Q_u(\varphi)\ge0$. Hence $u$ is stable.\qed

Since $K e^v\in L_{\mathrm{loc}}^1(\RR^m)$, the same argument in
Euclidean polar coordinates shows that a radial solution $v$ of
\eqref{eq2} is stable in $\RR^m$ if and only if
\begin{equation}\label{eq:Euclidean-radial-stability}
    \int_0^\infty
    \left\{|\psi'(s)|^2-K(s)e^{v(s)}\psi(s)^2\right\}
    s^{m-1}\,ds\ge0
    \quad\forall\psi\in C_c^\infty((0,\infty)).
\end{equation}

We are now in a position to prove Corollaries~\ref{cor1} and
\ref{cor2}.

\textbf{Proof of Corollaries~\ref{cor1} and~\ref{cor2}.}
Fix an arbitrary $\chi\in C_c^\infty((0,\infty))$ and define
$\widetilde\chi(s)=\chi(\phi^{-1}(s))$. Using $s=\phi(r)$ and
\eqref{phi:der}, we obtain
\begin{align*}
&\int_0^\infty
 \left\{|\chi'(r)|^2-(\sinh r)^\alpha
 e^{u_\beta(r)}\chi(r)^2\right\}(\sinh r)^{n-1}\,dr
 \\
&\qquad=\int_0^\infty
 \left\{|\widetilde\chi'(s)|^2-K(s)e^{v_\beta(s)}
 \widetilde\chi(s)^2\right\}s^{m-1}\,ds.
\end{align*}
Since $\phi$ is a diffeomorphism, the map
$\chi\mapsto\widetilde\chi$ is a bijection of
$C_c^\infty((0,\infty))$ onto itself. Lemma~\ref{lem:radial-stability}
and its Euclidean analogue \eqref{eq:Euclidean-radial-stability}
therefore show that $u_\beta$ is stable if and only if $v_\beta$ is
stable.

By Theorem~\ref{Bcor}, $v_\beta$ is stable if and only if
\begin{align*}
    v_\gamma(r)<v_\beta(r)
    \quad\text{for all }r\ge0\text{ and every }\gamma<\beta.
\end{align*}

If the solution family has structure (I), then every $v_\beta$
intersects a solution $v_\gamma$ with $\gamma<\beta$, so no
$v_\beta$ is stable. 
Under structure
(II), the initial ordering
$v_\gamma(0)<v_\beta(0)$ is preserved whenever $\gamma<\beta$,
so every solution is stable.
Suppose now that structure (III) holds with threshold $\beta^*$.
If $\beta\le\beta^*$, then $v_\beta$ does not intersect any
$v_\gamma$ with $\gamma<\beta$, and hence it is stable. If
$\beta>\beta^*$, choose $\gamma\in(\beta^*,\beta)$. Since
$v_\gamma$ and $v_\beta$ intersect, $v_\beta$ is unstable.
 The conclusions now follow from Theorems~\ref{thm1} and~\ref{thm2}.\qed

\appendix
\section{Euclidean separation results}
In this section, we introduce some known results on radial solutions which satisfy 
\begin{align*}
\begin{cases}
        v''(r)+\frac{m-1}{r}v'(r)+K(r)\exp(v(r))=0,\quad r>0,\\
        v(0)=\beta,
        \end{cases}
\end{align*}
where $K\in C((0,\infty))$ satisfies $K(r)>0$ for $r\in(0,\infty)$, $
\int_{0}^1rK(r)\,dr<\infty$, and $r^2K(r)\to0$ as $r\to0^+$.
\begin{theorem}\cite[Theorem~1.3]{BN}\label{BNthm1}
Let $m<10+4\ell_0$ with $\ell_0>-2$. Suppose that
\begin{equation}\label{Eucli:asy0}
    \lim_{r\to0^+} r^{-\ell_0}K(r)=k_0\quad \text{for some } k_0>0.
\end{equation}
Then either (I) or (III) holds.
\end{theorem}
\begin{theorem}\cite[Corollary~1.1]{BN}\label{BNcor1}
Let $m\ge 10+4\ell$ with $\ell>-2$. Assume that $r^{-\ell}K(r)$ is nonincreasing for $r\ge r_0$ with some $r_0>0$. Then either (II) or (III) holds.
\end{theorem}
\begin{theorem}\cite[Theorem~1.2,Proposition~4.1]{B}\label{BNthmC}
    Let $m<10+4\ell$ for some $\ell>-2$. Assume that $K(r)$ satisfies 
    \begin{align*}
        \lim_{r\to\infty}r^{-\ell}K(r)=k_\infty\quad \text{for some } k_\infty>0.
    \end{align*}
    Then (I) holds.
\end{theorem}
\begin{theorem}\cite[Theorem~1.5]{B}\label{Bthm1}
    Let $m\ge 10+4\ell$ for some $\ell>-2$. Assume that $K(r)$ satisfies
\begin{equation}\label{nonincreasing}
    r^{-\ell}K(r)\le \delta K_\ell(r)
    \quad\text{for } r>0,
\end{equation}
where $K_{\ell}(r):=\inf_{0<s<r}s^{-\ell}K(s)$ and $\delta=\frac{m-2}{4(2+\ell)}\ge1$. Then (II) holds.
\end{theorem}
\begin{theorem}\cite[Proposition~5.1]{B19}\label{Bcritical}
Let $m=10+4\ell$ with $\ell>-2$. Assume that
$r^{-\ell}K(r)$ is nondecreasing on $(0,\infty)$ and
\begin{align*}
    \lim_{r\to\infty}r^{-\ell}K(r)=k_\infty
    \quad\text{for some }k_\infty>0.
\end{align*}
If
\begin{align*}
    \int_1^\infty
    \bigl(k_\infty-r^{-\ell}K(r)\bigr)
    r^{(m-4)/2}\,dr<\infty,
\end{align*}
then either {\rm (II)} or {\rm (III)} holds.
\end{theorem}
\begin{theorem}\cite[Corollary~6.1]{H2}\label{Bcor}
Let $\beta\in\RR$. Then $v_\beta$ is stable if and only if $v_\gamma(r)<v_\beta(r)$ for all $r\ge0$ and every $\gamma\in(-\infty,\beta)$.   
\end{theorem}

\end{document}